\documentclass[12pt]{amsart}
\usepackage[utf8]{inputenc}
 \usepackage{xcolor}
\newcommand{\norm}[2]{\left\|#1\right\|_{#2}}
\newcommand{\scalar}[2]{\langle{ #1},{#2} \rangle}

\newcommand{\set}[1]{\left\{#1\right\}}

\newcommand{\lr}[1]{\left(#1\right)}

\newcommand{\real}{\mathbb R}

\newcommand{\domain}{\mathcal D}
\newcommand{\range}{\mathcal R}
\newcommand{\N}{\mathcal N}
\newcommand{\Ad}{A^\dag}
\newcommand{\xast}{x_\ast}
\newcommand{\pr}{\pi_{\overline{\range(A)}}}

\newtheorem{theorem}{Theorem}
\newtheorem{lemma}{Lemma}
\newtheorem{proposition}{Proposition}

\theoremstyle{remark}
\newtheorem{remark}{Remark}
\newtheorem{assumption}{Assumption}

\newtheorem{definition}{Definition}

\title[Regularization of the metric generalized
inverse]{Regularization of the metric generalized inverse in Banach
  spaces and the dichotomy phenomenon}
\author{Peter Math\'e}
\address{Weierstra{\ss} Institute for Applied Analysis and Stochastics, Mohrenstr. 39, 10117 Berlin,  Germany}
\email{peter.mathe@wias-berlin.de}
\author{Bernd Hofmann}
\address{Faculty of Mathematics, Chemnitz University of Technology,
 09107 Chemnitz,  Germany}
\email{hofmannb@mathematik.tu-chemnitz.de}
\date{\today}
\begin{document}

\begin{abstract}
For a bounded linear operator acting between Banach spaces, its metric generalized inverse is the analog to the prominent Moore-Penrose inverse for operators acting between Hilbert spaces.
This generalized inverse is well-defined for Banach spaces that are strictly convex and reflexive. Previous studies had been restricted to closed range situations, where the metric generalized inverse constitutes a continuous homogeneous mapping. The focus of the present study is on the ill-posed situation in the sense of Nashed, when the governing operator has a non-closed range.
We define and analyze iterative schemes as the Landweber and the Schulz-Newton method, as well as parametric schemes with focus on a specific Tikhonov method. Both types are called regularizations aimed at approximating the metric generalized inverse. As a fundamental feature of such schemes we observe a dichotomy. 
This emphasizes that these schemes, when applied to elements of the domain of the metric generalized inverse, approximate the corresponding best approximate solutions well, whereas the resulting approximations will be asymptotically unbounded if they are applied to elements that do not belong to the domain of the metric generalized inverse.
\end{abstract}
\maketitle
\begin{center}
{\sl In honor of our highly esteemed colleague Professor Zuhair Nashed, on the occasion of his 90th birthday. }
\end{center}   
\section{Introduction}
\label{sec:intro}

We are to consider operator equations
\begin{equation}
  \label{eq:3}
  A x=y,\quad x\in X,
\end{equation}
for a \emph{bounded linear} operator~$A \colon X \to Y$, acting between
\emph{Banach spaces} $X$ and $Y$. Our goal is to `solve' the inverse equation, which means
the reconstruction of $x= A^{-1}y \in X$ from $y\in Y$.
This problem is called \emph{ill-posed in the sense of Hadamard} if the operator~$A$ is
not boundedly invertible. One reason may be that a non-trivial
null-space~$\N(A)\subset X$ is present. The more interesting case is that the
range~$\range(A) \subset Y$ is not closed. Then the problem is called \emph{ill-posed in the sense of Nashed} \cite{Nashed86}.

For operator equations~\eqref{eq:3} in Hilbert spaces $X$ and $Y$ the most prominent single-valued substitute for $A^{-1}$ is the
\emph{Moore-Penrose generalized inverse}~$\Ad\colon \domain(\Ad) \subset Y \to X$.
This operator $\Ad$ is linear in a Hilbert space setting and \emph{bounded (continuous)} if 
the range $\range(A)$ of the operator $A$ is closed, in which case we have $\domain(\Ad)=Y$. Otherwise, for $\range(A) \not=\overline{\range(A)}$
the linear operator $\Ad$ is only densely defined on $Y$ and \emph{unbounded (not continuous)}, which requires
\emph{regularization} for treating the ill-posedness.  This means
that we are to approximate the unbounded map~$\Ad$ by a family of
bounded maps. While in Hilbert spaces the most studied regularization
is of spectral type, see the monograph~\cite{MR1408680}, for problems
in Banach spaces other types of regularization are natural. The most
important one is \emph{Tikhonov regularization} as a variant of variational regularization, see the
monograph~\cite{MR2963507} for details.

If $A: X \to Y$ is a bounded linear mapping between two Banach spaces $X$ and $Y$, different types of (often nonlinear) \emph{generalized inverses} $\Ad$ to $A$ exist. They also suffer from the possibly occurring ill-posedness of equation \eqref{eq:3} in the sense that $\Ad$ fails to be continuous. Then the goal of regularization is to build a family of (linear or nonlinear) continuous maps to approximate the respective generalized inverse $\Ad\colon \domain(\Ad) \subset Y \to X$.

The notion of metric generalized inverse, as an analog to the Moore-Penrose inverse, for operator equations in Banach space
is less studied. In Hilbert space every closed subspace is
complemented, i.e.,\ there is a bounded linear projection onto it.
However, in a Banach space setting, both the closed null-space $\N(A)$ of the bounded linear operator $A$ 
as well as the closure of the range~$\overline{\range(A)}$ may be \emph{uncomplemented},
and we refer to \cite{Flemmingbuch18} and \cite{HofKin25} for consequences of the uncomplementedness. 
A comprehensive account for concepts of generalized inverse mappings is given in~\cite{Nashed76} and \cite{NashedVotruba76}. Still linear projections play a major role.
Here, it is not assumed to have appropriate linear projections. Instead,
metric projections are considered, and our focus is on \emph{metric generalized inverses} $\Ad\colon \domain(\Ad) \subset Y \to X$
defined below in Definition~\ref{def:Ad}.
Given $y\in Y$, and a closed
subspace~$U \subset Y$, we are to find~$\pi_{U}(y)\in U$ which
minimizes the distance~$\norm{y - \pi_{U}(y)}{Y} = \inf_{z\in U}\norm{
y - z}{Y}$. These metric projections  exist whenever the subspace
is a \emph{Chebyshev subspace}, which is weaker than requiring the subspace
to be complemented. Metric projections exist whenever the Banach
space under consideration is reflexive and strictly convex. Therefore, we make the
following global assumption with respect to properties of the Banach spaces $X$ and $Y$. 
\begin{assumption}[Banach space geometry]\label{ass:geometry}
We consider a bounded linear operator $A \colon X \to Y$, where both spaces~$X$ and~$Y$ are \emph{reflexive} and \emph{strictly convex} Banach spaces.
\end{assumption}
Under this assumption, every closed
subspace of such a Banach space is a Chebyshev subspace. We will no further
dwell into this, and refer to the study~\cite{MR2024990},
instead. Specifically we highlight that in this case the (nonlinear, but homogeneous) metric
projections~$\pi_{\overline{\range(A)}}$
and~$\pi_{\N(A)}$ exist and are single-valued. 
The \emph{metric generalized inverse} was defined in \cite{MR2024990}. Here we specify this  as follows.

\begin{definition}
  [metric generalized inverse]\label{def:Ad}
  Under Assumption~\ref{ass:geometry} let \linebreak $\N(A)\subset X$
  denote the closed null-space and~$\range(A) \subset Y$ the range of the bounded linear
  operator $A:X \to Y$ mapping between the Banach spaces $X$ and $Y$. Then $\N(A)$ and the closure~$\overline{\range(A)} $ of the range are Chebyshev
  subspaces of $X$ and $Y$ with metric projections~$\pi_{\N(A)}$
  and~$\pi_{\overline{\range(A)}}$, respectively. There exists a
  homogeneous mapping \linebreak $\Ad\colon\domain(\Ad)\subset Y\to X$ with the
  following properties
  \begin{align}
    A \Ad A  x&= A x,\quad x\in X,\\
    \Ad A \Ad y &= \Ad y,\quad y\in  \domain(\Ad),\\
    \Ad A x &= (I_{X} - \pi_{\N(A)}) x,\quad x\in X,
              \label{it:AMA}\\
    A \Ad y &= \pi_{\overline{\range(A)}}y,\quad y\in \domain(\Ad).\label{it:AAM}
  \end{align}
  This mapping is called the \emph{metric generalized inverse}.
\end{definition}
\begin{remark}\label{rem:domain}
  It is interesting to notice, that historically, in finite dimension, or in Hilbert space, the first two properties were
  initially used in the original study by R. Penrose~\cite{Penrose_1955}, whereas the
  last two properties are due to E.~H.~Moore~\cite{moore1920reciprocal}.
\end{remark}

The construction of the mapping~$\Ad$ can be given explicitly, by using the
decomposition of the space~$X$, namely as
\begin{equation}
  \label{eq:1}
  X = \N(A) + C_{X}(A).
\end{equation}

The right term $C_{X}(A)$ is a closed cone. In general it is not a
linear subspace. Its description is given in Proposition~\ref{prop:Ad} in
Appendix~\ref{sec:fundamentals}.
Here, it is important to notice that the mapping~$A \colon C_X(A) \to Y$
is injective, such that formally~$\lr{A|_{C_{X}(A)}}^{-1}\colon \range(A)
\to X$ exists as a homogeneous single-valued mapping. So we can write
\begin{equation}\label{eq:Ad-construct}
\Ad y = \left[\lr{A|_{C_{X}(A)}}^{-1}\circ \pr\right](y),\ y \in \domain(\Ad).
\end{equation}
Additional details are also given in Proposition~\ref{prop:Ad} of Appendix~\ref{sec:fundamentals}.

Let us look at the
decomposition of the space~$Y$ as
\begin{equation} \label{eq:Ydeco}
Y = \overline{\range(A)} + C_{Y}(A),
\end{equation}
where the part~$C_{Y}(A)\subset Y$ constitutes a closed
cone. Again, further details are given in Appendix~\ref{sec:fundamentals}.
The domain and range of~$\Ad$ attain the forms
\begin{align}
\domain(\Ad) &= \range(A) + C_{Y}(A),
\intertext{and}
\range(\Ad) &= C_X(A).
\end{align}

The seminal paper~\cite{MR2024990} and several follow-up studies focus on
the closed range situation~$\range(A) = \overline{\range(A)}$, where the operator equation~\eqref{eq:3} is well-posed in the sense of Nashed~\cite{Nashed86}, but it still may be ill-posed in the sense of Hadamard if the operator~$A$ is non-injective. In this case, the metric generalized inverse $\Ad$ is defined on the whole space $\domain(\Ad)=Y$, and under weak additional assumptions it is even \emph{continuous}, see ~\cite[Theorem~4.1]{MR2024990}.

However, in the present study attention is drawn to the case of a \emph{non-closed range}, where
$\range(A) \not= \overline{\range(A)}$ and the operator equation \eqref{eq:3} is ill-posed in the sense of Nashed.
As seen from~\eqref{eq:Ad-construct} the mapping $\Ad$ is a homogeneous (in general nonlinear) operator, which is as a composition of a nonlinear (homogeneous) continuous projection operator $\pr$, and the inverse of the  restriction
of the injective operator $A$ to the part $C_{X}(A)\subset X$. This inverse and hence the whole metric generalized inverse $\Ad$ is
\emph{not continuous}, due to the non-closed range of $A$.

\par
In this study we are to approximate the metric generalized inverse by
means of a sequence of mappings, or by a parametric family with focus on Tikhonov regularization,
respectively. The imposed conditions will be called regularization. The present setup is in contrast to solve approximately the ill-posed
operator equation~\eqref{eq:3}, for which the monographs \cite{Scherzerbuch09} and \cite{MR2963507}
yield comprehensive details. 

When, for a given~$y\in Y$, we use a regularizing sequence~$X_{k}y,\,k=1,2,\dots$, 
or a regularizing parametric family $x_\alpha(y),\,\alpha>0$, we 
face two problems. If~$y\in\domain(\Ad)$ then the approximations should converge to the existing best approximate solutions~$\xast(y):= \Ad y$. To this end
 we have to take care that asymptotically~$C_{Y}(A)$ is suppressed. On the other hand, for~$y\not\in\domain(\Ad)$ the approximations cannot converge. Instead, under the made assumptions, they will necessarily `explode' (in norm). We refer to this as the \emph{dichotomy phenomenon}. Such dichotomy cannot be seen for operators~$A$ with a closed range $\range(A)$, because there are no elements $y \notin \domain(\Ad)$. Such dichotomies have been observed for spectral regularization of operators acting between Hilbert spaces, and we highlight~\cite[Thm.~5.2]{MR859375}. Within the context of Tikhonov regularization the unboundedness in norm of the parametric approximations~$x_\alpha(y)$, as~$\alpha\to 0$, for data~$y$ which do not have a minimum norm solution, and hence which are not in~$\domain(\Ad)$, has been proven in~\cite[Prop.~3.2]{Binder94}.

\par

We claim that the approximation of the metric generalized inverse (as a mapping) is
harder than the standard situation of solving (in a stable approximate manner) an ill-posed operator equation~\eqref{eq:3}, 
where noisy data $y^{\delta}$ to an \emph{exact right-hand side} $y \in \range(A)$ with $\|y^\delta-y\|_Y\le \delta$ and noise level $\delta>0$ can be used. For the metric generalized inverse~$\Ad$, noisy data $y^{\delta}=y+\xi$
with $\|\xi\|_Y \le \delta$ may perturb \emph{arbitrary elements} $y \in \domain(\Ad)$ that do not necessarily belong to $\range(A)$.
\smallskip

  The study is arranged as follows. In
  Section~\ref{sec:regularization} we introduce the concept of
  \emph{sequential regularization} to fix the requirements for the
  approximations~$X_{k}\colon Y \to X,\ k=0,1,2,\dotsc\,.$
  Section~\ref{sec:tikhonov-scheme-as}, however, introduces the counterpart of \emph{parametric regularization}.
   For both cases, we highlight that the conal part~$C_{Y}(A)\subset Y$
  will be suppressed, and we exhibit a dichotomy, distinguishing the
  behavior for the approximations, depending on
  whether~$y\in\domain(\Ad)$ or not.
  We complement the analysis with some discussion on the special case of Hilbert spaces in Section~\ref{sec:discussion}. Two appendices complete the paper. On the one hand, Appendix~\ref{sec:fundamentals} provides background information, with links to auxiliary studies. On the other hand, Appendix~\ref{sec:proof} gives a detailed proof of  Theorem~\ref{thm:dichotomyparam}.

\section{Sequential regularization: From Schulz to Landweber}
\label{sec:regularization}

There is vast literature for approximating the inverse of a matrix by
a sequence of mappings, probably starting with G. Schulz in 1933,
see~\cite{https://doi.org/10.1002/zamm.19330130111}, which we recall
in~\S~\ref{sec:second-order-scheme}, below. For a recent survey we refer to~\cite{MR3969100}, and  we mention~\cite{MR5022356}, putting such schemes within the framework for approximation of the (bounded) metric generalized inverse.
To the best of our knowledge, such iterative
schemes were not considered for equations that are ill-posed in the sense of Nashed, except in Hilbert spaces.
\par
Based on the
properties given in the items~\eqref{it:AMA} and~\eqref{it:AAM} of Definition~\ref{def:Ad} we
would require that those properties hold
asymptotically. It will be seen that an additional strengthening is
useful. We know from item~\eqref{it:xast3} of Proposition~\ref{prop:xast} 
in Appendix~\ref{sec:fundamentals} that~$\Ad y =
\Ad \pr y,\ y\in\domain(\Ad)$. For~$y\in\range(A)$ this trivially
holds true. However, taking into account the decomposition of~$Y$
from~\eqref{eq:Ydeco} we aim to require that the part
$C_Y(A)$ is suppressed, at least asymptotically.
Therefore we propose the following
\begin{definition}[sequential regularization]\label{def:seq-reg}
  Given a bounded linear operator~$A\colon X \to Y$, we call a
  sequence~$X_{k}\colon Y \to X,\ k=1,2,\dots,$  of homogeneous mappings a regularization, if for $k \to \infty$
  \begin{align}
   X_{k}y &\longrightarrow \Ad y,\quad y\in \domain(\Ad)\label{it:VkA},\\
  \intertext{and}  A X_{k}y &\longrightarrow  \pr y,\quad y\in Y\label{it:AVk}.
  \end{align}
\end{definition}

\begin{remark}\label{rem:CYA}
For $y\in C_Y(A)$, we have that $\pr y=0 \in \range(A)$ and hence $y \in \domain(A^\dagger)$ as well as $\xast(y)=0$. Consequently, we have~$\Ad y =0$ whenever~$y\in C_Y(A)$. This is analogous to the phenomenon of the Moore-Penrose generalized inverse 
$\Ad: \range(A)+\N(A^*)\subset Y \to X$ in Hilbert spaces, where also $\Ad y=0$ for $y \in \N(A^*)$. 
\end{remark}

The following observation shows that regularizations suppress the
information from $C_Y(A)$:
\begin{proposition}
  \label{kegel-weg}
  Let~$X_k y,\ k=0,1,2,\dots$ be any sequential regularization in the sense of Definition~\ref{def:seq-reg}. 
 For $y\in Y$, 
  we have that \linebreak $X_{k}\lr{y - \pr y} \to 0$ as~$k\to \infty$. Thus the conal part is asymptotically suppressed.
\end{proposition}
\begin{proof}
 As mentioned in Remark~\ref{rem:CYA}, we have for all $y \in Y$ that $y - \pr y \in C_Y(A) \subset \domain(\Ad)$ and that $x_*(y - \pr y)=0$. From~(\ref{it:VkA}) we deduce further 
  $$
 X_{k}\lr{y - \pr y} \longrightarrow \Ad\lr{y - \pr y} =0.
 $$
 This completes the proof.
\end{proof}

Based on this the following result shall be established.
\begin{theorem}\label{thm:dichotomy}
  Suppose that~$A\colon X\to Y$ is a bounded operator
  in a reflexive strictly convex Banach space~$Y$. Let~$X_{k},\
  k=0,1,\dots,$ be any regularization. \\
  Either~$y\in\domain(\Ad)$, and there is the
  minimum norm solution~$\xast(y)$,
  then we have that
  $$
  X_{k}y  \longrightarrow \xast(y) \quad \mbox{as} \quad k \to \infty;
  $$
  or~$y\not\in\domain(\Ad)$ and then
  $$\norm{X_{k} y}{X}\to\infty \quad \mbox{as} \quad k \to \infty.$$
\end{theorem}
\begin{proof}
  The part concerning $y\in\domain(\Ad)$ coincides exactly with the
  requirement~(\ref{it:VkA}), because~$\xast(y) = \Ad y$ whenever~$y\in\domain(\Ad)$.
\par
Now, we assume contrarily that the sequence~$X_{k}y,\ k=1,2,\dots,$ is uniformly norm bounded. Since this space~$X$ is assumed to be reflexive, it has a weakly convergent
subsequence by the Eberlein-Shmulyan Theorem. Without loss of generality, we assume the weak convergence $X_{k}y \rightharpoonup z\in
X$ in the Banach space $X$. Since the operator~$A$ is weak-to-weak continuous, we find that~$A
X_{k}y \rightharpoonup Az$, weakly in $Y$. But the left-hand side converges by the regularizing property~(\ref{it:AVk}) to $\pi_{\overline{\range(A)}}y$, which is not in $\range(A)$ for $y\not\in\domain(\Ad)$. Since the weak limit is
unique, we deduce that~$Az=\pi_{\overline{\range(A)}}y$. This contradicts
the assumption for the second assertion, and thus completes the proof of the theorem.
\end{proof}
To continue, we look at the requirements~(\ref{it:VkA})
and~(\ref{it:AVk}).
Specifically, for~$y\in \range(A)$ the convergence in~(\ref{it:VkA})
is towards~$\Ad A \xast = \lr{I_{X} - \pi_{\N(A)}}\xast$.
Thus, we consider the mappings
\begin{align}
S_{k}x&\colon = \lr{I_{X} - \pi_{\N(A)} -
          X_{k}A}x,\quad x\in X,\\
  \intertext{and}
  R_{k}y&\colon =
\pi_{\overline{\range(A)}}y - A X_{k}y,\quad y\in Y.
\end{align}
We state the following properties.
\begin{lemma}\label{lem:SkRk}
The following holds true.
\begin{enumerate}
    \item\label{it:spi} We have that~$S_k \pi_{\N(A)} x =0$,
    \item\label{it:Rkrange} $R_ky \in\overline{\range(A)}$ for~$y\in Y$, and
    \item\label{it:ASSR} we have the identity~$A S_k x = R_k A x,\ x\in X$.
\end{enumerate}
\end{lemma}
\begin{proof}
The first two items are seen, immediately.
    Finally, 
\begin{align*}
A S_kx &= A\lr{x - \pi_{\N(A)}x} - A X_k A x
= A x - A X_k A x\\
&= \pr A x - A X_k Ax  = R_k A x,\quad x\in X.    
\end{align*}
The proof is complete.
\end{proof}

We are to find conditions on the
sequence~$X_{k},\ k=0,1,2,\dots$ that allow for the sequences~$R_{k}, k=1,2,\dots$ and~$S_{k}, k=1,2,\dots$ to tend to zero.
\par

Before continuing, we emphasize the following observation. Even if we guarantee that the iterations~$X_k$ map into the conal part $C_X(A)=J_X\lr{\N(A)^\perp}$, then the misfit~$S_k$ might not tend to zero. This is the case, because this conal part is not a  linear subspace, hence having~$\lr{I_X - \pi_{\N(A)}}x\in C_X(A) $, and~$X_k A x\in C_X(A)$ will not guarantee that~$S_k x \in C_X(A)$, unless the conal part constitutes a linear subspace. The latter is the case if either the operator~$A$ is injective, or the space~$X$ is a Hilbert space.
Therefore, when considering sequential regularization we confine to the case that the projection~$x \to \pi_{\N(A)}x$ is linear, covering both cases mentioned before. Under this assumption we have that~$\pi_{\N(A)} S_k =0$, provided that~$\pi_{\N(A)} X_k y=0$.

The convergence
analysis will be based on two facts. First, 
a certain invariance property of the initial guess~$X_0$ is assumed. Then, we require a smallness
condition for~$k=0$. In this case it will be seen that there is a contraction
as~$k$ grows.
\begin{assumption}[Smallness condition]\label{ass:smallness}
  Suppose that~$X_{0}$ is such that
  \begin{enumerate}
  \item\label{it:invariance}
    \begin{enumerate}
    \item  we have that~$\pi_{\N(A)}X_{0}y=0,\ y\in Y$, and
    \item $X_{0}\pi_{\overline{\range(A)}}y=X_{0}y,\ y\in Y$.
    \end{enumerate}
  \item \label{it:small} For each~$x\in X,\ \norm{x}{X}\leq 1$, and~$y\in Y,\ \norm{y}{Y}\leq 1$, there is~$0< q <1$, such
    that
    \begin{align}
    \norm{\lr{I_{X} -\pi_{\N(A)} + X_{0}A}x}{X} &\leq q\label{eq:6}\\
     \intertext{and}
\norm{\lr{\pi_{\overline{\range(A)}} - A X_{0}}y}{Y}&\leq q.
   \end{align}
  \end{enumerate}
\end{assumption}
  \begin{remark}
  The first two items are trivially fulfilled whenever the
  operator~$A$ is injective and has a dense range, because in such a case
$\pi_{\N(A)}x=0$ for $x\in X$ and $\pi_{\overline{\range(A)}}y=y$ for $y \in Y$.
\end{remark}
Next we  will consider specific sequences~$X_{k},\ k=1,2,\dots$. Again, we mention
the recent study~\cite{MR5022356}, where results, similar to the ones
below are given. However, proved for operators with closed range, and
for \emph{quasi-additive} mappings~$\Ad$. The latter holds true
exactly when the metric projection~$\pi_{\N(A)}$ is linear.
\subsection{The Schulz-Newton method}
\label{sec:second-order-scheme}

The following scheme is classical and had been published by G.~Schulz in 1933,
\cite{https://doi.org/10.1002/zamm.19330130111} for invertible
matrices (with trivial null space) in the form~$X_{k+1} = X_{k}(2I_Y - A X_{k}),\ k=0,1,\dots$.
This translates to
\begin{equation}
  \label{eq:schulz}
  X_{k+1}y := X_{k}y + X_kR_k y,\quad y\in Y.
\end{equation}
The following lemma holds true.
\begin{lemma}
  \label{lem:Vk2}
  Suppose that the projection~$\pi_{\N(A)}$ is linear.
  Under Assumption~\ref{ass:smallness}  the following holds true
   \begin{enumerate}
   \item\label{it:pnV} We have, for all $k \in \mathbb{N}$, that $\pi_{\N(A)}X_{k}y=0,\ y\in Y$; and
   \item\label{it:Vpr} for all $k \in \mathbb{N}$ that $X_{k}\pr y = X_{k}y,\ y\in Y$.
        \item\label{it:k+1} 
   Moreover, $S_{k+1} = S_{k}^{2}$ and $R_{k+1}= R_{k}^{2},\,  k=0,1,2,\dots\,.$\label{it:order2}
   \end{enumerate}
 \end{lemma}
 \begin{proof}
     To verify~\eqref{it:pnV} and~\eqref{it:Vpr} we use induction. The assertion is true
   for~$k=0$. Suppose that it holds true for some~$k>0$. 
  Under the hypothesis for~$k$ we have that~$X_k y = X_k\pr y$, which in turn yields that~$R_ky = R_k \pr y$. 
  This implies
  $$
   X_{k+1}\pr y = X_k\pr y + X_kR_k\pr y  = X_{k+1} y.
  $$
  Similar reasoning applies to~$\pi_{\N(A)}X_ky$, provided that~$\pi_{\N(A)}$ is linear. This in turn yields~$\pi_{\N(A)} S_kx=0,\ x\in X$.
   \par
   For the assertion in item~\eqref{it:order2} we use the definition of~$S_k x$, making use of item~\eqref{it:ASSR} of Lemma~\ref{lem:SkRk}, that
   \begin{align*}
     S_{k+1}x &= x  - \pi_{\N(A)}x - X_{k+1}Ax\\
             &= \lr{x - \pi_{\N(A)}x  } - \lr{X_{k} + X_{k}R_k}Ax\\
             &=  \lr{x - \pi_{\N(A)}x -X_{k}Ax} - X_{k}R_kAx\\
             &=  \lr{x - \pi_{\N(A)}x -X_{k}Ax} - X_{k}A S_kx\\
     &= S_{k}x - X_{k}AS_k x\\
     &= S_{k}x  - \pi_{\N(A)}S_k x - X_{k}AS_kx\\
     &= S_{k}^{2}x.
   \end{align*}
   Again, the proof for~$R_{k+1}$ is similar.
 \end{proof}

\subsection{Landweber iteration: From Hilbert space to Banach space}
\label{sec:relat-spectr-regul}

The metric generalized inverse in Banach spaces is the extension of the Moore--Penrose
generalized inverse in Hilbert spaces. This means that the metric generalized inverse~$\Ad$ coincides
with the Moore-Penrose inverse when~$X$ and~$Y$ are Hilbert spaces,
because the defining properties are identical. Regularization is
required whenever the operator equation~(\ref{eq:3}) is
ill-posed. Besides spectral regularization, as for instance Tikhonov
regularization, iterative schemes gained interest, see the
monograph~\cite{MR1408680} for details. The monograph~\cite{MR859375}
is devoted to a systematic study of such
schemes, and Theorem~\ref{thm:dichotomy} above mimics Theorem~5.2 ibid.
\par
It seems interesting to discuss the relation of those iterative
schemes in Hilbert space with those considered here. To be specific we
consider Landweber iteration for an operator~$T\colon H \to G$ acting
between Hilbert spaces. Then the iterates, starting with some
initial~$x_{0}$, and data~$y\in G$, as
\begin{equation}
  \label{eq:landweber}
  x_{n+1} := x_{n} + \mu T^{\ast}\lr{y - T x_{n}},\quad n=0,1,\dots
\end{equation}
We readily see that the residuals~$y - T x_{n}$ play a prominent
role. This is in accordance with the mappings~$R_{k}$ from
above, see the second order scheme, studied before. However, and this seems crucial, the adjoint
operator~$T^{\ast}\colon G =G^{\ast} \to H=H^{\ast}$ is used to
map the residual information to the space~$H$. In order to extend such
iteration to a Banach space setting the duality mappings are used, and
a corresponding iterative scheme was developed and analyzed
in~\cite[Chapt.~6]{MR2963507}.
Within the present context of strictly convex and reflexive
  Banach spaces the dual mapping is continuous and single-valued. Therefore, we propose to replace the adjoint operator~$T^{\ast}$ of
  the Hilbert space setting, using the dual 
  operator~$A^{\prime}\colon Y^{\prime} \to X^{\prime}$ of the
  present~$A \colon X \to Y$, by
\begin{equation}
  \label{eq:9}
  B:= \lr{J_{X}}^{-1} A^{\prime} J_{Y}\colon Y \to X,
\end{equation}
in the Banach space setting. Above, the spaces~$X^\prime$ and~$Y^\prime$ denote the dual spaces to~$X$ and~$Y$, respectively. Then the version of the Landweber
iteration from~(\ref{eq:landweber}) reads as
\begin{equation}
  \label{eq:LandweberB}
  X_{k+1}y := X_{k}y + \mu B\lr{y - A {X_{k}y}} = X_{x}y + \mu B R_{k}y.
\end{equation}
\begin{remark}  
Comparing this iterative procedure with the Schulz-Newton method from~\eqref{eq:schulz}, started with~$X_0y = \mu B y$, we recognize the following: Landweber iteration updates the residuals~$R_ky$ always with the initial~$X_0:=\mu B$, whereas the Schulz-Newton method uses the current~$X_k$, instead.   This may lead to an improved performance of the Schulz-Newton method, and this will be manifested in Lemma~\ref{lem:landweber}, below. 
\end{remark}
Therefore, it makes sense to investigate the relevant properties of the mapping~$B$.
\begin{lemma}\label{lem:B}
  Let~$B$ be as in~(\ref{eq:9}). The following holds true.
  \begin{enumerate}
  \item\label{it:Bpi} We have that~$\pi_{\N(A)}B y=0,\ y\in Y$,
    \item \label{it:piB}$B \pr y = B y,\ y\in Y$.
  \item An element~$y\in Y$ fulfills~$A B y=0$ if and only if~$\pr y=0$.
  \end{enumerate}
\end{lemma}
\begin{proof}
  In the proof we use Proposition~\ref{prop:decomposition}.
  To prove item~\eqref{it:Bpi} it is required to show that~$B\lr{y -
    \pi_{\overline{\range(A)}}y} = 0$. Let $y\in C_Y(A)=
  J_{Y}^{-1}\lr{\range(A)^{\perp}}$,  we have to show that~$B y=0$.
 We notice that~$\range(A)^{\perp} = \N(A^{\prime})$, and hence there
  is~$y^{\prime}\in \N(A^{\prime})$ with~$y= J_{Y}^{-1}(y^{\prime})$.
  But
  $$
 B y = J_{X}^{-1} A^{\prime} J_{Y}(y) = J_{X}^{-1}
 A^{\prime} J_{Y}J_{Y}^{-1}(y^{\prime}) = J_{X}^{-1} A^{\prime}y^{\prime}= 0,
 $$
 which proves the first assertion.
 \par
 Similarly we proceed for item~\eqref{it:piB}. Due to the
 decomposition of~$Y$ it requires to show that~$\range(B)\subset
 J_{X}^{-1}\lr{\N(A)^{\perp}} =
 J_{X}^{-1}\lr{\overline{\range(A^{\prime})}}$.
 This is immediate from the definition of~$B$ in~(\ref{eq:9}).
 \par
 It remains to prove the last equivalence.
 Let us first suppose that \newline $\pr y =0$. By using the decomposition of~$Y$ this
 means that~$y \in J_{Y}^{-1}\lr{\overline{\range(A)}^{\perp}}$, which is
 equivalent to~$y\in J_{Y}^{-1}\lr{\N(A^{\prime})}$, and
 henceforth~$J_{Y}y \in \N(A^{\prime})$. It follows
 that~$A^{\prime}J_{Y}y=0$, which yields~$By=0$.
 \par
 To prove the converse, let~$y$ be such that~$ABy=0$. This means
 that~$z:= J_{X}^{-1}A^{\prime}J_{Y}y \in\N(A)$. By definition we have
 that
 $$
\scalar{z }{A^{\prime}J_{Y}y} = \norm{z}{X}^{2}
\norm{A^{\prime}J_{Y}y}{X^{\prime}}^{2} = \scalar{ A z }{J_{Y}y} =0.
$$
Therefore, it must hold true that~$A^{\prime}J_{Y}y=0$, i.e.,\
$J_{Y}y\in \N(A^{\prime}) = \overline{\range(A)}^{\perp}$. We see
that~$y\in J_{Y}^{-1}\lr{\N(A^{\prime})}=J_{Y}^{-1}\lr{
  \overline{\range(A)}^{\perp}} $, and hence in the decomposition
of~$Y$ its projection~$\pr y$ equals zero. The proof is complete.
\end{proof}
\begin{lemma}
  \label{lem:landweber}
      Let~$X_{0}:= \mu B$, and let
      $$
      R_{0}:=
      \pi_{\overline{\range(A)}} - \mu A B\;\; \text{ and
        accordingly}\;\; S_{0} :=
      I_{X} - \pi_{\N(A)} - \mu B A.
      $$
       Suppose that the projection~$\pi_{\N(A)}$ is linear.
      Then~$R_{k+1}:= R_{0}R_{k}$, and correspondingly~$S_{k+1}= S_{0}S_{k}$.
\end{lemma}
\begin{proof}
Applying the operator~$A$ to both sides in~(\ref{eq:LandweberB}) yields the identity
\begin{equation}
  \label{eq:11}
  A X_{k+1}y = A X_{k}y + \mu A B R_{k}y,
\end{equation}
which implies that
\begin{equation}
  \label{eq:12}
  R_{k+1} = \lr{\pi_{\overline{\range(A)}}- \mu A B} R_{k} = R_{0}R_{k},
\end{equation}
since~$\pi_{\overline{\range(A)}} R_{k}= R_{k}$.
A similar reasoning applies when
using~$y:= A x$ in~(\ref{eq:LandweberB}). By doing so, and using the
identity~$R_{k}A = A S_{k}$ we see that
$$
A S_{k+1} x= R_{k+1} Ax = R_k^2 Ax = R_k A S_k x= A S_k^2x. 
$$
This yields~$A S_{k+1} = A S_k^2$. Since under the made assumptions both~$S_k$ and~$S_{k+1}$ map into the conal part~$J_X(\N(A)^\perp)$, on which the mapping~$A$ is injective, we deduce that~$S_{k+1} = S_k^2$.
This proves both  assertions, and this completes
the proof.

\end{proof}
\begin{remark}
 By comparing the alterations of the $R_{k+1}$ with respect
  to $R_{k}$, we have a constant factor, here, whereas the Schultz-Newton 
  scheme from Subsection~\ref{sec:second-order-scheme}
  exhibits a power type
  behavior. Thus, under Assumption~\ref{ass:smallness}, the latter shows
  higher order decay.
\end{remark}
 We turn to the main result in this section.
 \begin{theorem}\label{prop:AAM}
  Suppose that the projection~$\pi_{\N(A)}$ is linear.
   Under Assumption~\ref{ass:smallness} for the initial approximations~$X_0$, the sequences~$X_{k},\
   k=0,1,2,\dots$ from~(\ref{eq:schulz}) 
   and~(\ref{eq:LandweberB})  above constitute regularizations in the sense of Definition~\ref{def:seq-reg}.
 \end{theorem}
\begin{proof}
   We first observe that under Assumption~\ref{ass:smallness} both
  sequences~$S_{k}y$ and~$R_{k}y$ for~$k=0,1,2,\dots$ tend to zero.
  Indeed, we have that \linebreak $\norm{S_{0}}{X \to
   X}\leq q < 1$. By iterating along the lines of item~\eqref{it:order2} in Lemma~\ref{lem:Vk2}, we find for $x$ with $\norm{x}{X}\leq 1$ that we have convergence
 $$
\norm{S_{k}x}{X \to X} = \norm{S_{0}^{2^{k}}x}{X \to X}\leq q^{2
  ^{k}}\to 0.
$$
For Landweber iteration the contraction is linear, i.e.,\
$\norm{S_{k}x}{X \to X} \leq q^{k}$.
The convergence proof for the sequences~$R_{k}y,\ k=0,1,2,\dots$ is similar, and hence omitted.

\par

  This means that we have for $y\in\range(A)$, precisely with~$y=Ax$, that
  $$
  X_{k}y = X_{k}A x \longrightarrow \lr{I_{X} - \pi_{\N(A)}}x = \Ad Ax
  = \Ad y = \xast(y).
  $$
  To continue we highlight that the schemes under consideration
  obey $X_{k}y = X_{k}\pr y,\ y\in Y$.
  Then, for~$y\in\domain(\Ad)$ and hence with~$\pr y \in\range(A)$, the
  above convergence yields, due to item~\eqref{it:xast3} of
  Proposition~\ref{prop:xast} as
  $$
 X_{k}y = X_{k}\pr y \longrightarrow \xast(\pr y) = \xast(y),
 $$
 which completes the proof of~(\ref{it:VkA}). The proof
 of~(\ref{it:AVk}) is immediate from the convergence of the
 sequences~$R_{k}y,\ k=0,1,2,\dots$ to zero.
\end{proof}
We mention that other schemes are known to exhibit the regularizing
properties. A general guideline for the construction of such schemes
is given in~\cite[Thm.~3.13]{MR5022356}.

\section{Parametric regularization: The Tikhonov scheme}
\label{sec:tikhonov-scheme-as}
In contrast to approximations of the metric generalized inverse $A^\dagger: \domain(A) \subseteq Y \to X$ by means of sequences of homogeneous linear operators~$X_k: Y \to X$ presented before,
we consider in this section Tikhonov regularization in Banach spaces as an example of a parametric regularization scheme.
To this end we adopt the notion of regularization from
Definition~\ref{def:seq-reg}, accordingly.
\begin{definition}
  [parametric regularization]\label{def:parreg}
  Given a bounded linear operator~$A\colon X \to Y$, we call a
  parametric family~$x_{\alpha}\colon Y \to X,\ \alpha>0$, regularization, if for $\alpha \to 0$
  \begin{align}
   x_{\alpha}(y) &\longrightarrow \Ad y,\quad y\in \domain(\Ad)\label{it:xaA},\\
    \intertext{and}  A x_{\alpha}(y) &\longrightarrow \pi_{\overline{\range(A)}}
               y,\quad y\in Y\label{it:Axa}.
  \end{align}
\end{definition}

\begin{theorem} \label{thm:Parathm}
    Suppose that~$A \colon X \to Y$ acts between Banach spaces that obey Assumption~\ref{ass:geometry}. Let~$x_\alpha(y)$ be a parametric regularization.\\
  Either~$y\in \domain(\Ad)$,  and there is a unique minimum norm solution $x_*(y)=A^\dagger y$, then we have that 
  \begin{equation} \label{eq:normcon-dichotomy}
   x_\alpha(y) \longrightarrow x_*(y) \quad \mbox{as} \quad \alpha \to 0,
  \end{equation}  
  or~$y\not\in \domain(\Ad)$ then 
  \begin{equation}\label{eq:normdiv}
      \norm{x_\alpha(y)}{X} \to \infty\quad \text{as}\quad \alpha\to 0.
  \end{equation}
\end{theorem}

\medskip

Theorem~\ref{thm:Parathm} represents the counterpart to Theorem~\ref{thm:dichotomy}, where again the dichotomy phenomenon occurs. The proof of Theorem~\ref{thm:Parathm} is analog to the proof of Theorem~\ref{thm:dichotomy} and was therefore omitted.

\medskip

In the following we will establish that the Tikhonov scheme, with details given below,  constitutes a parametric regularization.
There is vast literature on variational regularization for operators
acting between Banach spaces, and we mention the basic study \cite{HKPS07} and the monographs~\cite{MR2963507} as well as \cite{Scherzerbuch09}. However, the focus of these works is on the stable approximate solution of equation~\eqref{eq:3} (and extensions to nonlinear forward operators) with \emph{attainable data}, where either the exact right hand side $y \in \range(A)$ to Equation~\eqref{eq:3} or noisy data $y^\delta$ of $y \in \range(A)$ with $\|y^\delta-y\|_Y \le \delta$ are given.  Our goal, however, is to recover
the metric generalized inverse $A^\dagger$ in the sense that elements $x_*(y)=A^\dagger y$ are to be approximated by convergent sequences of Tikhonov-regularized solutions for all elements $y \in \domain(A^\dagger)=\range(A)+C_Y(A)$.
This includes the case of \emph{non-attainable} data when $y \in \domain(A^\dagger)$, but $y \notin \range (A)$.
In this case, $y$ does not play the role of a right-hand side to Equation \eqref{eq:3}, because there is a distance
$${\rm dist}(y,\overline{\range(A)})=\|y-Ax_*(y)\|_Y=\|y-\pi_{\overline{\range(A)}} y\|_Y>0,$$
in contrast to $y =\pi_{\overline{\range(A)}} y\in \range(A)$, where ${\rm dist}(y,\overline{\range(A)})=0$.

\par

The present scheme will be based on the specific Tikhonov functional
\begin{equation} \label{eq:Tikfct}
T_\alpha(x,y):=\|Ax-y\|_Y^p+\alpha \|x\|_X^q, \quad \alpha>0,
\end{equation}
which is similarly discussed in \cite[Chapter~5]{MR2963507}. Here, we consider exponents $p,q \ge 1$.
Moreover, we denote by $x_\alpha(y) \in X$ the minimizers of the functional \eqref{eq:Tikfct} over all $x \in X$,
for fixed \emph{regularization parameter} $\alpha>0$.
Under
Assumption~\ref{ass:geometry}, this functional is strictly convex for
all $\alpha > 0$, because the misfit term $\|Ax-y\|_Y^p$ is convex due to the linearity of $A$ and the part $\alpha\,\|x\|_X^q$ is strictly convex for all $\alpha>0$ due to the strict convexity of the Banach space $X$.
\par

For obtaining assertions on norm-convergence for minimizers $x_\alpha(y)$ of the Tikhonov functional \eqref{eq:Tikfct},
which we call \emph{regularized solutions}, we will need in addition to reflexivity and strict convexity of the Banach space $X$ that it is a \emph{Radon-Riesz space}, or in other words, $X$ possesses the \emph{property~H} (see also item (iii) of Proposition~\ref{prop:Ad} in Appendix~\ref{sec:fundamentals}). This means that
weak convergence~$x_n \rightharpoonup x$ in $X$  and convergence of the norms~$\norm{x_n}{X} \to \norm{x}{X}$ yield convergence in norm $\|x_n-x\|_X \to 0$ as $n \to \infty$. 

\begin{remark}\label{rem-RR}
We note that \emph{uniform convexity} as
imposed on the Banach space $X$ implies reflexivity and strict convexity as well as the property~H for that space. Reflexivity occurs as a result of the Milman-Pettis Theorem, see~\cite[Chapt.~5.2]{YosidaFA}. The fact that every uniformly convex spaces has the Radon-Riesz property is due to M. I. Kadets. 
Thus, the results in this section hold true for uniformly convex spaces.
\end{remark}

\begin{proposition}\label{pro:Tikexist}
Under Assumption~\ref{ass:geometry} there is a uniquely determined minimizer $x_\alpha(y) \in X$ of the Tikhonov functional \eqref{eq:Tikfct} for all $y \in Y$ and $\alpha>0$. If moreover~$X$ is assumed to be a Radon-Riesz space, then
the minimizer $x_\alpha(y)$ (regularized solution) is stable with respect to perturbations in $y \in Y$. Precisely, for any
sequence $(y_n)_{n=1}^\infty \subset Y$, the condition
$\lim_{n \to \infty}\|y_n-y\|_Y$ implies that $\lim_{n \to \infty}\|x_\alpha(y_n)-x_\alpha(y)\|_X=0$.
\end{proposition}
\begin{proof}
As is well-known, continuous linear operators between Banach spaces
are also weak-to-weak continuous. Moreover, for the reflexive Banach
space $X$, the penalty functional $x \to \|x\|_X^q$ is for $p \ge 1$ a
proper convex, lower semi-continuous and stabilizing functional in the
sense of \cite[Ass.~3.22]{MR2963507}. The stabilizing property of the
penalty means that sub-level sets $\{x \in X: \norm{x}{X} \le c\}$ are weakly compact in $X$ for all $c\ge 0$. Then  Propositions~4.1 and 4.2 from \cite{MR2963507} apply for $F:=A$ and yield the assertions of the proposition when taking into account that the uniqueness of regularized solutions is a consequence of the strict convexity of the Tikhonov functional.
The transfer from weak to norm convergence, when applying~\cite[Prop.~4.2]{MR2963507}, is based on the Radon-Riesz property of the Banach space $X$ by reminding us that the occurring penalty term in the Tikhonov functional~\eqref{eq:Tikfct} is a power of the norm in $X$.
\end{proof}
It was shown for sequential regularization in Section~\ref{sec:regularization} that such schemes suppress the conal part, see Proposition~\ref{kegel-weg}. Similar is valid for parametric regularization. Actually, an even stronger assertion holds true for Tikhonov regularization.
\begin{proposition}\label{prop:kegel-weg-tikhonov}
Given~$y\in Y$ let~$x_\alpha(y)$ be the minimizer of the functional~\eqref{eq:Tikfct}.
    If~$y\in C_Y(A)$ 
    then we have $x_\alpha(y)=0$, regardless of~$\alpha>0$.
\end{proposition}
\begin{proof}
   It was mentioned in Remark~\ref{rem:CYA}, for~$y\in C_Y(A)$ with $\pr y=0$, that we have $\xast(y)=0$. 
    Using this information we can bound the functional~\eqref{eq:Tikfct} as
    \begin{align*}
    \norm{A x_\alpha(y) - y}{Y}^p + \alpha\norm{x_\alpha(y)}{X}^q
    &\leq \norm{A \xast(y) - y}{Y}^p + \alpha\norm{\xast(y)}{X}^q \\
    &= \norm{\pr y - y}{Y}^p =\left({\rm dist}(y,\overline{\range(A)})\right)^p.
    \end{align*}
    Hence~$A x_\alpha(y)$ is the unique element in $Y$ that approximates $y\in C_Y(A)$ in the best way over all elements from the range of $A$, which implies the equation $A x_\alpha(y) = \pr y=0$. This yields~$x_\alpha(y) = \xast(y)=0$ for all $\alpha>0$
    and all $y \in C_Y(A)$. Now the proof is complete.
\end{proof}
It will be shown in the due course of Theorem~\ref{thm:dichotomyparam} that this phenomenon, mentioned above in 
Proposition~\ref{prop:kegel-weg-tikhonov} extends 
from $y \in C_Y(A)$ to $y\in \domain(\Ad)$ in an asymptotic manner. Instead of~$x_\alpha(y)= \xast(y)$ for~$y\in C_Y(A)$, we have there that~$x_\alpha(y) \to \xast(y)$ as $\alpha \to 0$ for all $y\in\domain(\Ad)$, not only for $y \in \range(A)$.
This result, as part of Theorem~\ref{thm:dichotomyparam}, is novel within the context of the relevant literature, and it adds a new aspect to the properties of Tikhonov regularization.

The following result, the proof of which is given in Appendix~\ref{sec:proof}, 
constitutes that, under additional geometric assumptions, the Tikhonov scheme represents a parametric regularization in the sense of Definition~\ref{def:parreg}.

\begin{theorem}\label{thm:dichotomyparam}
  For the bounded linear operator $A\colon X\to Y$, let $X$ and $Y$ be
  Banach spaces satisfying Assumption~\ref{ass:geometry}. Moreover, $X$  and $Y$ are assumed to be a Radon-Riesz spaces. Consider the uniquely determined minimizer $x_\alpha(y) \in X$ of the Tikhonov functional~\eqref{eq:Tikfct} for
  fixed $y \in Y$ and regularization parameters $\alpha>0.$ The following holds true. \begin{enumerate}
      \item  For~$y \in \domain(A^\dagger),$ and hence when there is a unique minimum norm solution $x_*(y)=A^\dagger y$, we have norm-convergence in $X$ of the form
  \begin{equation} \label{eq:normcon}
   x_\alpha(y) \longrightarrow x_*(y) \quad \mbox{as} \quad \alpha \to 0.
  \end{equation}  
      \item For all $y \in Y$ we have the convergence
      \begin{equation} \label{eq:Anormcon}
  Ax_\alpha(y) \longrightarrow \pi_{\overline{\range(A)}} y\quad \mbox{as} \quad \alpha \to 0.
  \end{equation}  
  \end{enumerate}
  Thus, the family~$x_\alpha(y),\ \alpha>0$, constitutes a parametric regularization in the sense of Definition~\ref{def:parreg} .
  \end{theorem}

\begin{remark} \label{rem:Tiknoise}
For attainable data~$y$, with ${\rm dist}(y,\overline{\range(A)})=0$, an extension of the result \eqref{eq:normcon} to the case of noisy data $y^\delta=Ax_*(y)+\xi \in Y$ with $\|\xi\|_Y \le \delta$
and noise level $\delta>0$ can be obtained for the related Tikhonov functional
\begin{equation} \label{eq:Tikfctnoise}
T^\delta_\alpha(x,y^\delta):=\|Ax-y^\delta\|_Y^p+\alpha \|x\|_X^q, \quad \alpha>0,
\end{equation}
with associated minimizers $x_\alpha^\delta(y):= x_\alpha(y^\delta)$, under Assumption~\ref{ass:geometry} and for Radon-Riesz spaces $X$.
For the a priori choice of the regularization parameter~$\alpha=\alpha(\delta)$, with
$$ \alpha(\delta) \to 0 \quad \mbox{and} \quad \frac{\delta^p}{\alpha(\delta)}\to 0, \;\; \mbox{as} \quad \delta \to 0,
$$
it is shown that 
$$
x^\delta_{\alpha(\delta)}(y) \to x_*(y).
$$
This can even be extended to appropriate nonlinear forward operators $F:\domain(F) \subset X \to Y$. We refer in this context for details to~\cite[Theorem~3.5]{HKPS07}, 
\cite[Theorem~3.26]{Scherzerbuch09} and \cite[Corollary~4.6]{MR2963507}.

\par
For non-attainable data $y=y_1+y_2 \in Y$ with $y_1 \in \range(A)$, $y_2 \in C_Y(A)$ and ${\rm dist}(y,\overline{\range(A)})>0$, convergence results under noise are only known for Hilbert spaces $X$ and $Y$. The following is shown in~\cite[Theorem~3.5]{Binder94} (again also with possible extensions to appropriate nonlinear forward operators $F$).
For $p=q=2$ in \eqref{eq:Tikfctnoise}, we have
\begin{equation} \label{eq:convnoiseHilbert} x^\delta_{\alpha(\delta)}(y) \to x_*(y)  \quad \mbox{for}\quad \alpha(\delta) \to 0 \;\; \mbox{and} \;\; \frac{\delta}{\alpha(\delta)}\to 0 \;\; \mbox{as} \;\; \delta \to 0 .
\end{equation} 
\end{remark}

\section{The case of Hilbert spaces} \label{sec:discussion}
It is worthwhile to highlight the approximation of the metric generalized inverse in case that the operator~$A\colon X \to Y$  acts between Hilbert spaces.
This
includes an inspection of the specific behaviour of the Tikhonov regularization, of Landweber iteration and of the Schulz-Newton method in Hilbert spaces. 

 In this case we know that the duality mappings~$J_X$ and~$J_Y$, see Appendix~\ref{sec:fundamentals} for the definition,  are the identities in~$X$ and~$Y$, respectively, by identifying the dual space~$X^\prime$ with~$X$, and similarly~$Y^\prime$ with~$Y$. 
By notifying that we have that~$\N(A)^\perp = \overline{\range(A^\ast)}$, and similarly~$\overline{\range(A)}^\perp = \N(A^\ast)$, for the adjoint operator~$A^\ast\colon Y \to X$ to~$A$.
The decompositions (direct sums) are thus
\begin{align*}    
 X &= \N(A) + \overline{\range(A^\ast)}
 \intertext{and}
 Y &= \overline{\range(A)} + \N(A^\ast)
\end{align*}
These are even orthogonal decompositions into two closed linear subspaces. Also, the projections~$\pi_{\N(A)}$ as well as~$\pr$ are orthogonal projections. 

The restricted injective inverse mapping $A^{-1}:\range(A) \to \overline{\range(A^\ast)}$ is also linear. Consequently, the Moore-Penrose generalized inverse $A^\dagger: \domain(A^\dagger)=\range(A)+\N(A^*) \to X$ is a linear operator with range $\range(A^\dagger)=\overline{\range(A^\ast)} \perp \N(A).$
In this case, the metric generalized inverse coincides with the Moore-Penrose inverse~$\Ad\colon Y \to X$ with~$\domain(A^\dagger) = \range(A) + \N(A^\ast)$. For $y \in\domain(A^\dagger)$ there is a uniquely determined minimum norm solution $x_*(y)\perp \N(A)$. The important fact is that for all~$y \in \domain(\Ad)$, with~$y = y_1 + y_2$, and~$y_1 \in \range(A)$ and~$y_2\in \N(A^\ast)$, we have that~$A^*y = A^\ast y_1$. Thus, regularization schemes, which are based on using~$A^\ast y$, automatically suppress the part in~$\N(A^\ast)$.

Let us specify this for Landweber iteration as studied 
in Subsection~\ref{sec:relat-spectr-regul}. Notice that here the governing map~$B\colon Y \to X$ equals~$A^\ast$, and hence in Hilbert spaces this yields reconstructions, linear in the data~$y$. Using initially~$X_0y:= \mu B y$, we are in the situation from above. Item~\eqref{it:Bpi} in Lemma~\ref{lem:B} highlights that again we see~$By \perp \N(A)$. This led to Theorem~\ref{prop:AAM}, which asserts that Landweber iteration constitutes a sequential regularization in the sense of Definition~\ref{def:seq-reg} when started with the mapping~$B$, provided that the parameter~$\mu>0$ is small enough.

Similar reasoning applies to Tikhonov regularization (see, e.g.,~\cite[Chapt.~3]{Groetsch84}).
Hilbert spaces are of Radon-Riesz type, satisfy Assumption~\ref{ass:geometry}, and the uniquely determined minimizer $x_\alpha(y) \in X$ of the Tikhonov functional
\begin{equation} \label{eq:TikfctHilbert}
T_\alpha(x,y):=\|Ax-y\|_Y^2+\alpha \|x\|_X^2, \quad \alpha>0,
\end{equation}
is linear and it is explicitly given as
\begin{equation} \label{eq:RegsolHilbert}
x_\alpha(y)=(A^*A+\alpha I)^{-1}A^*y, \quad y \in Y.
\end{equation}

Taking into account that $A^*y=0$ for $y \in \N(A^*)$,
one easily finds from \eqref{eq:RegsolHilbert} the limit statements
$$ \lim_{\alpha \to 0} x_\alpha(y) = x_*(y)=A^\dagger y, \quad y \in \domain(A^\dagger)$$
and
$$ \lim_{\alpha \to 0} A x_\alpha(y) = \pi_{\overline{\range(A)}}y, \quad y \in Y, $$
which coincide with the conditions~\eqref{it:xaA} and \eqref{it:Axa} in Definition~\ref{def:parreg}.

A proof of the convergence result \eqref{eq:convnoiseHilbert} with noisy data in both the attainable data and the non-attainable data case is rather simple for linear operators~$A$ in the Hilbert space setting, because of having the explicit representation~
$$x^\delta_\alpha(y)=(A^*A+\alpha I)^{-1}A^*y^\delta,$$
in combination with the fact that using $A^*y^\delta$ automatically suppresses the part of $y^\delta$ in $\N(A^*)$.
\smallskip

Overall, the above reasoning applies to general \emph{spectral regularization} in Hilbert space. We will not dwell into this. However, the regularizing properties, as given in Theorems~\ref{thm:dichotomy} and Theorem~\ref{thm:dichotomyparam}, are well known there, and we mention the monograph~\cite[Thm.~5.2]{MR859375}.

\smallskip

We conclude this section with a brief discussion, highlighting the fundamental differences between Landweber iteration and the Schulz-Newton method.  Landweber iteration is known to be a gradient descent algorithm, i.e., \ mimimizing the quadratic functional
$$
x \longrightarrow \frac{\norm{A x - y}{Y}^2}{2},\quad x\in X.
$$
In this case the gradient evaluates to~$- A^\ast y$, which in turn yields Landweber iteration in the form of~\eqref{eq:landweber}, after introducing an acceleration factor~$\mu>0$.
The Schulz-Newton iteration is also a gradient descent algorithm, but based on sequentially minimizing the functional
$$
 x \longrightarrow \frac{\norm{ \lr{ I_X- X_k A} x}{X}^2}{2}, \quad x\in X.
$$
Notice that we have to constrain minimization to elements~$x \in \N(A)^\perp$. This corresponds to the conal part in the decomposition of~$X$, which however simplifies in the case that the operator~$A$ is injective.
It is immediately seen that this results in~\eqref{eq:schulz}. The latter scheme is known to be quadratically convergent, provided that the initial guess~$X_0$ fulfills the smallness Assumption~\ref{ass:smallness}.

\appendix

\section{Fundamentals }
\label{sec:fundamentals}
In this appendix we collect several facts which are underlying the
considerations in the main body.
\paragraph{{\bf Duality mappings}}
On the Banach space~$X$ we let~$J_{X}\colon X \to X^{\prime}$ denote the
  map
  $$
J_{X}(x) := \set{x^{\prime}\in X^{\prime},\ \scalar{x^{\prime}}{x} =
  \norm{x}{X}^{2}= \norm{x^{\prime}}{X^{\prime}}^{2}}.
$$
Similar applies to the Banach space~$Y$.
\begin{remark}
  There is parametric family of duality maps, often denoted
  by~$J_{p}^{X}, p\geq 1$. Here we confine to~$p=2$ to ensure the
  \emph{homogeneity} of the duality map, and hence let~$J_{X}:= J_{2}^{X}$.
  For a comprehensive study on duality maps we refer to~\cite{MR1079061}. It is immediate from the definition that the duality map~$J_X$ is the identity when the space~$X$ is a Hilbert space.
\end{remark}
The following is known.
\begin{proposition}\label{prop:duality}
  Suppose that the Banach space~$X$ is strictly convex and reflexive.
  The following holds true.
  \begin{enumerate}
  \item The duality map is
    single-valued, continuous, and homogeneous
    $J_{X}(\lambda x) = \lambda  J_{X}(x),\ x\in X$, and for all $\lambda \in\real$.
  \item We have that~$J_{X^{\prime}}(J_{X}(x)) = x$, and
also~$J_{X}(J_{X^{\prime}}(x^{\prime}))= x^{\prime}$.
\item The inverse~$J_{X}^{-1}(x^{\prime})$ is single-valued and continuous for
  each~$x^{\prime}\in X^{\prime}$.
  \end{enumerate}
\end{proposition}
\paragraph{{\bf Decomposition of the Banach spaces}}
Given a closed subspace~$L\subset X$ with annihilator~$L^{\perp}\subset
X^{\prime}$, we let~$J_{X}^{-1}(L^{\perp}) = \set{x\in X,\ J_{X}(x)\cap
L^{\perp}\not=\emptyset}$. Since the duality mapping is homogeneous,
this set constitutes a cone in~$X$. It is a closed cone,
because~$L^{\perp}$ is closed, and the mapping~$J_{X}^{-1}$ is continuous.
We state the following results, see e.g.~\cite{MR1338551}. 
\begin{proposition}\label{prop:decomposition}
  Under Assumption~\ref{ass:geometry} the following holds true.
  Given the linear operator~$A\colon X \to
  Y$, we have the decompositions
  \begin{align}
    X &= \N(A) + J_{X}^{-1}\lr{\N(A)^{\perp}}\label{it:Xdec}
    \intertext{and}
    Y&= \overline{\range(A)} + J_{Y}^{-1}\lr{\range(A)^{\perp}}.\label{it:Ydec}
  \end{align}
  The mapping~$A\colon  J_{X}^{-1}\lr{\N(A)^{\perp}} \to \range(A)$ is injective.
\end{proposition}
In the introductory discussion we have used the abbreviations 
$$
C_X(A):= J_{X}^{-1}\lr{\N(A)^{\perp}},\quad  \text{and}~C_Y(A):= J_{Y}^{-1}\lr{\range(A)^{\perp}}.
$$
\medskip

\paragraph{{\bf Properties of the metric generalized inverse}}

Here we collect results from~\cite[Cor.~3.2 \& Thm.~4.1]{MR2024990}.
\begin{proposition}\label{prop:Ad}
  Under Assumption~\ref{ass:geometry} the following holds true.
  \begin{enumerate}
\item The
  metric generalized inverse~$\Ad$ exists.
\item Its domain~$\domain(\Ad)$ is given as
\begin{equation}
  \label{eq:domainAd}
   \domain(\Ad) = \range(A) \stackrel{\cdot}{+} J_{Y}^{-1}\lr{\range(A)^{\perp}}.
 \end{equation}
  In particular,
  $y\in\domain(\Ad)$ if and only if~$\pi_{\overline{\range(A)}}y  \in\range(A)$.
  \item If the space~$X$ has the property~H (Radon-Riesz property), i.e.,
    ~$x_{n}\rightharpoonup x$ and~$\norm{x_{n}}{X} \to \norm{x}{X}$,
    yields~$x_{n} \to X$ in norm, then the metric generalized inverse
    is continuous if and only if the range~$\range(A)\subset Y$
    is closed. In this case~$\domain(\Ad)=Y$.
    If the range of~$A$ is not closed, then the domain~$\domain(\Ad)$ is dense in~$Y$.
  \end{enumerate}
\end{proposition}
The last assertion could not be found in the literature, and we sketch a proof of~$\overline{\domain(A^\dagger)}=Y$, here:
Consider an arbitrary
element $y=y_1+y_2 \in Y$, for which $y_1 \in \overline{R(A)}$ and $y_2 \in C_Y(A)$ are uniquely determined. Then there is a sequence $z_n \in \range(A)$ converging in norm to $y_1$ as $n \to \infty$. Moreover, the sums~$z_n+y_2$ belong to the domain of $A^\dagger$, because~$\pr (z_n + y_2) = z_n\in\range(A)$. Finally we have $\|z_n+y_2-y\|_Y = \norm{z_n - y_1}{Y} \to 0$ as $n \to \infty$.

\medskip

\paragraph{{\bf Minimum norm solutions, best approximate solutions}}

Let~$y\in Y$ be any element. An element~$x_{0}\in X$ is called \emph{best
  approximate solution} to~$Ax =y$, if
$$
\norm{A x_{0} - y}{Y} = \inf\set{\norm{A x - y}{Y},\ x\in X}.
$$
Best approximate solutions need not exist. If they exist, then such a solution
is called \emph{minimum norm solution} whenever it has mimimal norm among all
best approximate solutions.
The following results can be seen from the proof of~\cite[Thm.~3.2]{MR2024990}, and from the fact
that $\pr y \notin \range(A)$ for all $y \notin \domain(\Ad)$.
\begin{proposition}\label{prop:xast}
  Under Assumption~\ref{ass:geometry} the following is known for the
  elements~$y\in\domain(\Ad)$.
  \begin{enumerate}
\item\label{it:xast} $\xast(y) := \Ad y$ is a best approximate
  solution to the equation $A x = y$.
\item \label{it:xast2} $\xast(y) := \Ad y$ is the uniquely determined minimum norm solution to the equation~$A x =
  \pi_{\overline{\range(A)}}y$.
  \item \label{it:xast3}  $\xast(y) := \Ad y$
  is the only solution to~$A x = \pr y$ with~$x\in
  J_{X}^{-1}\lr{\N(A)^{\perp}}$, and hence it satisfies~$\pi_{\N(A)}\xast(y)=0$.
  In particular we have that~$\xast(y)= \xast(\pr y)$ for~$y\in\domain(\Ad)$, since~$\pr$ is a projection.
 \end{enumerate}
For~$y\notin\domain(\Ad)$, however, best approximate solutions and minimum norm solutions do not exist.
\end{proposition}

\section{Proof of Theorem~\ref{thm:dichotomyparam}}
\label{sec:proof}

As basic regularizing property we have, for all $x \in X$ and all $y \in Y$, that $ T_\alpha(x_\alpha(y),y) \le T_\alpha(x,y)$,  which means
\begin{equation} \label{eq:mingen}
\|Ax_\alpha(y)-y\|_Y^p +\alpha\|x_\alpha(y)\|^q_X \le  \|Ax-y\|_Y^p +\alpha\,\|x\|_X^q. 
\end{equation}
By omitting the second term on the left-hand side and applying the superior limit for $\alpha \to 0$, we derive from
\eqref{eq:mingen}, for all $x \in X$ and all $y \in Y$,  the inequality
\begin{equation} \label{eq:suplimgen}
\limsup_{\alpha \to 0} \|Ax_\alpha(y)-y\|_Y \le \|Ax-y\|_Y,    
\end{equation}
and consequently, for all $y \in Y$, the estimate
\begin{equation} \label{eq:supliminf}
\limsup_{\alpha \to 0} \|Ax_\alpha(y)-y\|_Y \le \inf_{x \in X}\|Ax-y\|_Y=\|\pr y -y\|_Y= {\rm dist}(y,\overline{\range(A)}). 
\end{equation}
Because we have, for all $\alpha>0$, that $Ax_\alpha(y) \in \range(A) \subset \overline{\range(A)}$, this gives
$$\;\,{\rm dist}(y,\overline{\range(A)}) \le \liminf_{\alpha \to 0}\|Ax_\alpha(y)-y\|_Y\;,
$$
and hence with \eqref{eq:supliminf} that the limit $\;\lim_{\alpha \to 0}\|Ax_\alpha(y)-y\|_Y\;$ exists for all $y \in Y$ 
and obeys the equalities
\begin{equation} \label{eq:disteq}
\lim_{\alpha \to 0}\|Ax_\alpha(y)-y\|_Y={\rm dist}(y,\overline{\range(A)})=\|\pr y -y\|_Y \qquad \forall y \in Y \,.
\end{equation}
Evidently, for every $y \in Y$, $Ax_\alpha(y)$ is bounded in norm for all regularization parameters $\alpha>0$, and thus there exist weak limits $z \in Y$ for sequences
$Ax_{\alpha_n}(y) \rightharpoonup z$. Since the set $\overline{\range(A)}$ is convex and closed, hence weakly closed in $Y$,
we have $z \in \overline{\range(A)}$. 
Furthermore $\|z-y\|_Y \le \lim_{n \to \infty}\|Ax_{\alpha_n}(y)-y\|_Y
={\rm dist}(y,\overline{\range(A)})$.
Since~$z \in \overline{\range(A)}$, we see that~$z=\pr y$. 

Summarizing, we have weak convergence $Ax_{\alpha_n}(y)-y \rightharpoonup \pr y-y$, and by~\eqref{eq:disteq} we also  have convergence of the norms
$\|Ax_{\alpha_n}(y)-y\|_Y \to \|\pr y-y\|_Y$ as $n \to \infty$. By the Radon-Riesz property this gives $\|Ax_\alpha(y) - \pr y\|_Y \to 0$ as $ \alpha \to 0$ for all $y \in Y$. The arguments above show the validity of~\eqref{eq:Anormcon}.

To establish~\eqref{eq:normcon}, we restrict our considerations to the case~$y \in \domain(\Ad)$. There we have a uniquely determined minimum norm solution $x_*(y)$ to any fixed $y \in \domain(\Ad)$ with $Ax_*(y)=\pr y$. Also, taking into account \eqref{eq:disteq} we find that 
$$\lim_{\alpha \to 0} \|Ax_\alpha(y)-y\|_Y =\|Ax_*(y)-y\|_Y={\rm dist}(y,\overline{\range(A)}).$$ 
For $x:=x_*(y)$, \eqref{eq:mingen} attains the form
\begin{equation} \label{eq:minprop}
\|Ax_\alpha(y)-y\|_Y^p + \alpha\|x_\alpha(y)\|^q_X  \le  \|Ax_*(y)-y\|_Y^p +\alpha\,\|x_*(y)\|_X^q,
\end{equation}
The assertion \eqref{eq:normcon} 
can be derived from \eqref{eq:minprop} by rewriting this as
$$ \alpha\|x_\alpha(y)\|^q_X  \le \left[ \|Ax_*(y)-y\|_Y^p - \|Ax_\alpha(y)-y\|_Y^p \right]  +\alpha\,\|x_*(y)\|_X^q. $$
For all $\alpha>0$, the difference in the square brackets is non-positive. This is
a consequence of  $$\|Ax_*(y)-y\|_Y={\rm dist}(y,\overline{\range(A)}) \le  \|Ax_\alpha(y)-y\|_Y,$$ 
which is due to $Ax_\alpha(y) \in \range(A) \subset \overline{\range(A)}$.
Consequently, we obtain that
\begin{equation} \label{eq:xnorm}
\|x_\alpha(y)\|_X \le \|x_*(y)\|_X \quad \mbox{for all} \;\; \alpha>0,
\end{equation}
which in turn yields that
$\limsup_{\alpha \to 0}\|x_\alpha(y)\|_X \le  \|x_*(y)\|_X $. Since $X$ is reflexive, and because~$A$ is weak-to-weak continuous, for any sequence $\alpha_n \to 0$, there is some element $z \in X$ such that weak limits in $X$ and $Y$, respectively, apply as 
$$x_{\alpha_n}(y) \rightharpoonup z \quad \mbox{and} \quad A x_{\alpha_n}(y) \rightharpoonup Az \quad \mbox{as} 
\quad n \to \infty. $$ Then the weak lower semi-continuity of the norm provides us with
$$
\|z\|_X \le \liminf_{n \to \infty} \|x_{\alpha_n}(y)\|_X \le  \limsup_{n \to \infty}  \|x_{\alpha_n}(y)\|_X   \le \|x_*(y)\|_X.
$$
In view of~\eqref{eq:disteq} we see that
$$
\|Az-y\|_Y \le \liminf_{n \to \infty} \|Ax_{\alpha_n}(y)-y\|_Y = {\rm dist}(y,\overline{\range(A)}).   
$$
Since $Az \in \range(A)$, we even have $\|Az-y\|_Y={\rm dist}(y,\overline{\range(A)})$. Thus, every such limit element
$z \in X$ is a best approximate solution, obeying $\|Az-y\|_Y=\min_{x \in X} \|Ax-y\|_Y$ and  
$\|z\|_X \le \|x_*(y)\|_X$. This has the consequence that the weak limit elements $z$, for any sequences of $\alpha$ tending to zero, coincide with the uniquely determined minimum norm solution $x_*(y)$. 
We aim at using that $X$ is a Radon-Riesz space.
Therefore, to complete the proof of~\eqref{eq:normcon}, we still have to show that $\lim_{\alpha_n \to 0} \|x_{\alpha_n}(y)\|_X \to \|x_*(y)\|_X$ as $n \to \infty$. However, the weak convergence $x_{\alpha_n}(y) \rightharpoonup x_*(y)$ as $ n\to \infty$ implies with \eqref{eq:xnorm} that 
$$\|x_*(y)\|_X \le \liminf_{n \to \infty} \|x_{\alpha_n}(y)\|_X \le  \limsup_{n \to \infty} \|x_{\alpha_n}(y)\|_X \le \|x_*(y)\|_X $$
and hence $\lim_{n \to \infty} \|x_{\alpha_n}(y)\|_X=\|x_*(y)\|_X$, which completes the proof.

\end{document}